\documentclass{article}

\usepackage{booktabs}
\usepackage{graphicx}
\usepackage{algorithm}
\usepackage{amsmath}
\usepackage{amsfonts}
\usepackage{amssymb}
\usepackage{amsthm}
\usepackage{algorithmic}
\usepackage{subfig}
\usepackage{tabularx}
\usepackage{pdflscape}

\usepackage[utf8]{inputenc} 
\usepackage{array} 
\usepackage{paralist} 
\usepackage{verbatim} 

\usepackage{array}
\usepackage{longtable}
\usepackage{hyperref}

\newcommand\rothead[1]{\rotatebox{90}{#1}}
\usepackage{amsmath,amsfonts,latexsym,amssymb} 

\newcommand{\binomial}[2]{\ensuremath{\left(
\begin{array}{c} #1 \\ #2 \end{array} \right)}}
\newcommand{\boldalpha}{\boldsymbol{\alpha}}

\newcommand{\bydef}{\ensuremath{\stackrel{\mbox{\scriptsize def}}{=}}}
\newcommand{\bez}{\mathrm{B\acute ez}}

\renewcommand{\P}{{\mathbb P}}
\newcommand{\C}{\mathcal{C}}

\newcommand{\neww}[1]{}

{\bf}{} 
{\bf}{} 
{\bf}{}
\newtheorem{remark}{Remark}{\bf}{}
\newtheorem{theorem}{Theorem}{\bf}{}
\newtheorem{corollary}{Corollary}{\bf}{}
{\bf}{}
{\bf}{}
\newtheorem{definition}{Definition}

\begin{document}
\title{Power Flow as an Algebraic System}

\author{Jakub Marecek, Timothy McCoy, and Martin Mevissen
\thanks{J. Marecek and M. Mevissen are with IBM Research -- Ireland, 
IBM Technology Campus Damastown, Dublin D15, Ireland. 
e-mail: jakub.marecek@ie.ibm.com.
T. McCoy is with Google.
}
}

\maketitle

\begin{abstract}
Steady states of alternating-current (AC) circuits have been studied in considerable detail.
In 1982, Baillieul and Byrnes derived an upper bound on the number of steady states in a loss-less AC circuit
[IEEE TCAS, 29(11): 724--737] and conjectured that this bound holds for AC circuits in general.
We prove this is indeed the case, among other results, by studying a certain multi-homogeneous 
structure in an algebraisation. 
\end{abstract}



\section{Introduction}

For more than 60 years \cite{4074830}, steady states of alternating-current (AC) circuits have been studied in considerable detail.
The key problem, sometimes known as the power flow or load flow problem,
 considers complex voltages $V_k$ at all buses $k$ as variables, 
 except for one reference bus ($k = 0$), where the power supplied.
When one denotes complex admittance matrix $Y$, 
complex current $I_k$, and complex power $S_k$ at
bus $k$,
the steady-state equations are based on:
\begin{align}
S_k = V_k I_k^* = V_k \sum_{ l \in N} Y^*_{k,l} V_l^* = 
      \sum_{l \in N} Y^*_{k,l} V_k V_l^* 
\label{Sk}
\end{align}
where asterisk denotes complex conjugate. 
This captures the complex, non-convex non-linear nature \cite{932273,70472,260940,481635,Klos1991} of any problem in the AC model.
In order to obtain an algebraic system from \eqref{Sk}, 
one needs to reformulate the complex conjugate.
In order to do so, one may replace all $V_k^*$ with independent variables $U_k$, and filter for ``real'' solutions where $U_k=V_k^*=\Re{V_k}- \Im{V_k} \imath$ once the complex solutions are obtained.
Thereby, we obtain a particular structure, 
which allows us to prove a variety of results.

In particular, the main contributions of our paper are:
\begin{itemize}
 \item a reformulation of the steady-state equations to a multi-homogeneous algebraic system
 \item analytical results on the number and structure of feasible solutions considering losses, resolving a conjecture of Baillieul and Byrnes \cite{1085093}, which has been open for over three decades
 \item 	empirical results for some well-known instances, including the numbers of roots, conditions for non-uniqueness of optima, and tree-width.
\end{itemize}
Our analytical results rely on the work of Morgan and Sommese \cite{MORGAN-SOMMESE,Sommese2005} on multi-homogeneous structures.
Our empirical results rely on Bertini \cite{bates2013numerically}, a leading implementation of homotopy-continuation methods.
As we explain in Section \ref{relatedwork}, ours is not the first algebraisation of the system \eqref{Sk}, cf. \cite{4074830,4047153,4047382,1085093}, and there is a long history \cite{70552,230632,260891,1406189,Mehta2015a,Mehta2014,Mehta2015,6344759,7393554}
of the use of homotopy-continuation methods.  

\section{The Problem}
\label{sec:system}

In order to make the paper self-contained, we restate of the steady-state equations. 
Consider a circuit represented by an undirected graph $(N, E)$, where vertices $n \in N$ are called buses and edges $\{ l,m \} \in E \subseteq N \times N$ are called branches, 
and an admittance matrix $Y = G + B\imath \in \mathbb{C}^{|N| \times |N|}$, where the real part of an element is called conductance $G= (g_{lm})$ and the imaginary part susceptance $B=(b_{lm})$.
Each bus $k \in N$ is associated with complex voltage $V_k =\Re{V_k}+ \Im{V_k} \imath$, 
complex current $I_k
=\Re{I_k}+ \Im{I_k} \imath
$, and power $S_k = P_k + Q_k \imath$ demanded or generated.
Let $0\in N$ correspond to a reference bus
 with phase $\Im{V_0} = 0$ and magnitude $|V_0|$ fixed;
powers at all other buses are fixed too.
(In a variety of extensions, there are other buses, denoted generators, where 
voltage magnitude, but not the phase and not the power is fixed.)
Each branch $(l,m) \in E$ is associated with the complex power $S_{lm}=P_{lm}+ Q_{lm} \imath$.
The key constraint linking the buses is Kirchhoff's current law, which 
 stipulates the sum of the currents injected and withdrawn at each bus is 0.
Considering the relationship $I = YV$, 
the steady state equations hence are:
\begin{align}
P^g_k &= P_k^d+\Re{V_k} \sum_{i=1}^n ( { \Re{y_{ik}} \Re{V_i} - \Im{y_{ik}} \Im{V_i} }) \notag \\
    & \quad + \Im{V_k} \sum_{i=1}^n ({ \Im{y_{ik}} \Re{V_i} + \Re{y_{ik}} \Im{V_i} }) 
\label{eqn:Pk} 
\\
Q^g_k &= Q_k^d+\Re{V_k} \sum_{i=1}^n ({ - \Im{y_{ik}} \Re{V_i} - \Re{y_{ik}} \Im{V_i} }  ) \notag \\
    & \quad + \Im{V_k} \sum_{i=1}^n ({ \Re{y_{ik}} \Re{V_i} - \Im{y_{ik}} \Im{V_i} }) \\ 
P_{lm} &= b_{lm} ( \Re{V_l} \Im{V_m} - \Re{V_m} \Im{V_l}) 
\\
    & \quad+ g_{lm}( \Re{V_l}^2 +\Im{V_m}^2- \Im{V_l} \Im{V_m}- \Re{V_l} \Re{V_m}) \notag \\
Q_{lm}  &= b_{lm} ( \Re{V_l} \Im{V_m} - \Im{V_l} \Im{V_m}-\Re{V_l}^2 -\Im{V_l}^2) \notag \\
    & \quad + g_{lm}( \Re{V_l}\Im{V_m}- \Re{V_m} \Im{V_l}- \Re{V_m} \Im{V_l})  \notag \\
  & \quad -\frac{\bar b_{lm}}{2}(\Re{V_l}^2 +\Im{V_l}^2) 
\label{eqn:Qlm}
\end{align}


Additionally, one can optimise a variety of objectives over the steady states.
In one commonly used objective function, one approximates the costs 
of real power $P_0$ generated at the reference bus $0$
by a quadratic function $f_0$: 
\begin{align}
\label{eq:obj-costs}
\text{\bf cost} := f_0 (P_0).
\end{align}
(In a variety of extensions, in which there are other buses where the power is not fixed, 
there would be a quadratic function for each such bus and the quadratic function of power
would be summed across all of these buses.)
In the $L_p$-norm loss objective, one computes a norm of the vector $D$ obtained by summing apparent powers $S(u, v) + S(v, u) \forall (u, v) \in E$ for: 
\begin{align}
\label{eq:obj-loss}
|| D ||_p = \bigg( \sum_{(u, v) \in E} |  S(u, v) + S(v, u) |^p \bigg)^{1/p}.
\end{align}
The usual $|| D ||_1$ is denoted {\bf loss} below. 
We consider these objectives only in Section VII, while 
our results in Sections IV--VI apply independently of the use of any objective function whatsoever.

\section{Definitions from Algebraic Geometry}
\label{sec:defs}

In order to state our results, we need some definitions from algebraic geometry.
While we refer the reader to 
\cite{1085093,malajovich2007computing} for the basics, we present
the concepts introduced in the past three decades, not yet widely
covered by textbooks.
For a more comprehensive treatment, please see \cite{MORGAN-SOMMESE,Sommese2005,SHAFAREVICH}. 


Let $n \ge 0$ be an integer and $f(z)$ be a system of $n$ polynomial equations 
in $z \in \C^n$ with support $(A_1, \dots, A_n)$:
\begin{equation}
\label{eq:**}
\left\{
\begin{array}{lcl}
f_1(z) &=& \sum_{\boldalpha \in A_1} f_{1 \mathbf \boldalpha} 
z_1^{\alpha_1} z_2^{\alpha_2} \cdots z_n^{\alpha_n} \\
& \vdots & \\
f_n(z) &=& \sum_{\boldalpha \in A_n} f_{n \mathbf \boldalpha} 
z_1^{\alpha_1} z_2^{\alpha_2} \cdots z_n^{\alpha_n} \ ,\\
\end{array}
\right.
\end{equation}
where coefficients $f_{i \mathbf \boldalpha}$ are non-zero
complex numbers.
It is well known \cite{fulton1998} that the polynomials define $n$ projective hypersurfaces in a projective space 
$\C\P^n$. 

B\'ezout theorem states that either hypersurfaces intersect in an infinite set with some component of positive dimension, or the number of intersection points, counted with multiplicity, is equal to the product $d_1 \cdots d_n$, where $d_i$ is the degree of polynomial $i$.
We call the product $d_1 \cdots d_n$ the usual B\'ezout number.

The usual B\'ezout number can be improved by considering: 

\begin{definition}[Structure]
Any partition of the index set 
$\{1, \dots , n\}$ into $k$ sets 
$I_1, \dots, I_k$
defines a structure.
There, $Z_j = \{ z_i : i \in I_j \}$ is known as the group of variables
for each set $I_j$.
The associated degree $d_{ij}$ of a polynomial $f_i$ with respect to group $Z_j$ is 
\begin{align}
d_{ij} \bydef \max_{\boldalpha \in A_i} \ \sum_{l \in I_j}
\ \alpha_l. \label{eq:associateddegree}
\end{align}
We say that $f_i$ has \emph{multi-degree} $(d_{i1},\ldots,d_{in})$.
\end{definition}

Whenever for some $j$, for all $i$, the same $d_{ij}$ is attained for all $\boldalpha \in A_i$, we call the system 
(\ref{eq:**}) homogeneous in the group of variables $Z_j$.
The projective space associated to the group of variables $Z_j$ in a structure has dimension
\begin{align}
a_j \bydef \left\{
\begin{array}{ll}
|I_j| - 1& \text{if (\ref{eq:**}) is homogeneous in $Z_j$, and}
\\
|I_j| & \text{otherwise.}
\end{array}
\right.
\label{eq:asocdimension}
\end{align}

\begin{definition}[Multi-homogeneous B\'ezout Number]
Assuming $n = \sum_{j=1}^k a_j$, the multi-homogeneous B\'ezout number 
$\bez(A_1, \dots, A_n; I_1, \dots, I_k)$
is defined as the coefficient of the term
$\prod_{j=1}^k \zeta_j^{a_j}$,
where $a_j$ is the associated dimension \eqref{eq:asocdimension},
 within the polynomial $\prod_{i=1}^n \sum_{j=1}^k d_{ij} \zeta_j$, in variables $\zeta_j, j= 1 \ldots k$
 with coefficients $d_{ij}$ are the associated degrees \eqref{eq:associateddegree}; that is 
$(d_{11} \zeta_1 + d_{12} \zeta_2 + \ldots + d_{1k} \zeta_k)$
$(d_{21} \zeta_1 + d_{22} \zeta_2 + \ldots + d_{2k} \zeta_k) \cdots$
$(d_{2n} \zeta_1 + d_{2n} \zeta_2 + \ldots + d_{nk} \zeta_k)$.
\end{definition}

Consider the example of Wampler \cite{WAMPLER1992} in $x \in \C^3$:
\begin{equation}
\label{eq:**}
\left\{
\begin{array}{lcl}
p_1(z) &=&  x_1^2 + x_2 + 1,\\
p_2(z) &=&  x_1 x_3 + x_2 + 2,\\
p_3(z) &=&  x_2 x_3 + x_3 + 3,
\end{array}
\right.
\end{equation}
with the usual B\'ezout number of 8. 
Considering the partition $\{x_1, x_2\}$, $\{ x_3 \}$,
where $d_{11} = 2$, $d_{12} = 0, d_{21} = d_{22} = d_{31} = d_{32} = 1$.
the monomial
 $\zeta_1^{2} \zeta_2^{1}$
 is to be looked up in the polynomial 
$2 \zeta_1 (\zeta_1 + \zeta_2)^2$.
The corresponding multi-homogeneous B\'ezout number is hence 
4 and this is the minimum across all possible structures.  

In general, the multi-homogeneous B\'ezout number 
$\bez(A_1, \dots, A_n; I_1, \dots, I_k)$
is an upper bound on the number of
isolated roots of (\ref{eq:**}) in $\C\P^{a_1} \times \cdots \times \C\P^{a_k}$, 
and thereby an upper bounds the number of isolated finite complex roots of (\ref{eq:**}).
There are a variety of additional methods for computing the multi-homogeneous B\'ezout number, e.g., \cite{WAMPLER1992}.
In the particular case where $A = A_1 = \cdots = A_n$,
we denote  
\begin{equation}\label{eq:bezfromI}
\bez(A_1, \dots, A_n; I_1, \dots, I_k) \bydef 
\binomial {n}{a_1\ a_2\ \cdots\ a_k} 
\
\prod_{j=1}^k d_j^{a_j} \ ,
\end{equation}
where $d_j = d_{ij}$ (equal for each $i$) and the 
multinomial coefficient 
\[
\binomial {n}{a_1\ a_2\ \cdots\ a_k} 
\bydef
\frac{n!}{a_1!\ a_2!\ \cdots\ a_k!}
\]
is the coefficient of $\prod_{j=1}^k \zeta_j^{a_k}$
in $(\zeta_1 + \cdots + \zeta_k)^{n}$ with 
$n = \sum_{j=1}^k a_j$, as above. 

In summary, the multi-homogeneous B\'ezout number provides a sharper bound on the number of isolated solutions of a system of equations than the usual B\'ezout number $\prod_{i = 1}^{n} d_i = d_1 \cdots d_n$.
In the famous example of the eigenvalue problem \cite{Wampler1993}, it is known that 
the B\'ezout number is $2^n$, whereas there exists a structure with  
multi-homogeneous B\'ezout number of $n$.
We hence study the multi-homogeneous structure within the
steady state equations of alternating-current circuits.

\section{The Multi-Homogeneous Structure}
\label{sec:algebraic}

Notice that in order to obtain an algebraic system from the steady-state equations 
(\ref{eqn:Pk}--\ref{eqn:Qlm}), 
one needs to reformulate the complex conjugate.
In order to do so, one may replace all $v_n^*$ with independent variables $u_n$, and later filter for 
 solutions where $u_n=v_n^*$ once the complex solutions are obtained. We denote such solution ``real''.
Let $G$ be the set of slack generators for which $|v_n|$ is specified, and assume $0\in G$ corresponds to a reference node with phase $0$. 
Notice that the use of variables $v_n$ and $u_n$ produces a multi-homogeneous structure with variable groups $\{v_n\}$ and $\{u_n\}$:

\begin{align}
	& v_n\sum_k Y_{n,k}u_k + u_n\sum_k Y_{n,k}^*v_k = 2p_n	& n\in N \setminus G \notag \\
	& v_n\sum_k Y_{n,k}u_k - u_n\sum_k Y_{n,k}^*v_k = 2q_n	& n\in N \setminus G \notag \\
	& v_nu_n = |v_n|^2		& n\in G - \{0\} \notag \\
	& v_0 = |v_0|, \; u_0 = |v_0|	&
\label{eq:multihomo}
\end{align}

For example for the two-bus network, we obtain:

\begin{align}
	& v_1(Y_{1,0}u_0+Y_{1,1}u_1) + u_1(Y_{1,0}^*v_0+Y_{1,1}^*v_1) = 2p_1 \notag \\
	& v_1(Y_{1,0}u_0+Y_{1,1}u_1) - u_1(Y_{1,0}^*v_0+Y_{1,1}^*v_1) = 2q_1 \notag \\
	& v_0 = |v_0|, \; u_0 = |v_0|	&
\label{BertiniEqs}
\end{align}


Using the algebraic system, one can formulate a number of 
structural results concerning power flows.

\section{An Analysis for $s = 1$} 
\label{sec:properties}

For the particular multi-homogeneous structure, which is the partition of the variables into several groups in \eqref{eq:multihomo}, we can bound the number of isolated solutions:

\begin{theorem}
\label{thm:finiteRoots}
With exceptions on a parameter set of measure zero, the
alternating-current power flow \eqref{eq:multihomo} has a finite 
number of complex solutions, which is bounded above by:
\begin{align}
{2n - 2 \choose n - 1 }
\end{align}
\end{theorem}

\begin{proof}
Each equation in the system \eqref{eq:multihomo} is linear in the $V$ variables and also in the $V^*$ variables, giving rise to a natural multi-homogeneous structure of mult-idegree $(1,1)$.  Since the slack bus voltage is fixed at a reference value, the system has $2n-2$ such equations in $(n-1, n-1)$ variables.  By the multi-homogeneous form of B\'ezout's Theorem (see e.g. Theorem 8.4.7 in \cite{Sommese2005}), the total number of solutions in multi-projective space $\C\P^{n-1}\times\C\P^{n-1}$ is precisely the stated bound, counting multiplicity.  Some subset of these lie on the affine patch $\C^{n-1}\times\C^{n-1}\subset\C\P^{n-1}\times\C\P^{n-1}$, giving the result.
\end{proof}

Notice that this applies also to some well-known instances of alternating-current
optimal power flows (ACOPF). For example, the instances of Lesieutre et al. \cite{6120344} and Bukhsh et al. \cite{6581918} have only a single ``slack'' bus, whose active and reactive powers are not fixed,
and hence the result applies.
Notice that the exception of measure-zero set is necessary; cf. Example 4.1 of \cite{1085093}. 

As we will illustrate in the next section, this bound is tight in some cases.
Deciding whether the bound on the number of roots 
obtained using a particular multi-homogeneous structure is tight for a particular
instance is, nevertheless, hard. This can be seen from:

\begin{theorem}[Theorem 1 of Malajovich and Meer \cite{malajovich2007computing}]
There does not exist a polynomial time algorithm to approximate
the minimal multi-homogeneous B\'ezout number for a polynomial system \eqref{eq:**} 
up to any fixed factor, unless P = NP.
\end{theorem}

We can, however, show there exists a certain structure among these solutions:

\begin{corollary}
\label{thm:evenRoots}
If there exists a feasible solution of the alternating-current power flow, then the solution has even multiplicity greater or equal to $2$
or another solution exists.
\end{corollary}

\begin{proof}
The finite number of solutions to the power flows problem of Theorem~\ref{thm:finiteRoots} is even.  Observe that $(U, \hat{U})$ is a solution of the system \eqref{eq:multihomo} if and only if $(\hat{U}^*, U^*)$ is a solution. This implies that the non-``real'' solutions, that is solutions for which $U^*\neq\hat{U}$, necessarily come in pairs.  It follows that ``real'' feasible solutions are also even in number, counting multiplicity.  The result follows.
\end{proof}

Note that a solution having multiplicity greater than 1 is a special case that is highly unlikely in a real system. Moreover, it is easily detected, since the Jacobian at a solution is nonsingular if and only if the solution has multiplicity 1.

\section{Alternating-Current Optimal Power Flow} 

One may also make the following observations about the alternating-current optimal power flows, i.e., the problem of optimising an objective over the steady state:

\begin{remark}
\label{thm:asymptotically-convergent}
For the alternating-current power flows, where powers are fixed at all but the reference bus,  
whenever there exists a real feasible solution,
except for a parameter set of measure zero,
one can enumerate all feasible solutions in finite time.
\end{remark}

Indeed: By Theorem~\ref{thm:finiteRoots}, we know there exist a finite number of isolated solutions to the system \eqref{eq:multihomo}.
By the homotopy-continuation method of Sommese et al. \cite{Sommese2005,bates2013numerically}, we can enumerate the roots with probability 1, which allows us to pick the global optimum, trivially.
Notice that Bertini, the implementation of the method of Sommese \cite{bates2013numerically}, makes it possible to check that all roots are obtained.
Notice that the addition of inequalities 
 can be accommodated
by filtering the real roots.

Nevertheless, this method is not practical, as there may be too many isolated solutions to enumerate. 
Generically, this is indeed the case, 
whenever there are two or more generators 
with variable output, i.e., buses, whose 
active and reactive power is not fixed:

\begin{corollary}
\label{thm:positiveDimensional}
In the alternating-current optimal power flow problem, i.e., with $s > 1$, where powers are variable outside of the reference bus and there are no additional inequalities,  the 
complex solution set is empty or positive-dimensional, except for a parameter set of measure zero.
When the complex solution set is positive-dimensional, if a smooth real feasible solution exists, then there are infinitely many real feasible solutions.
\end{corollary}

\begin{proof}
For each slack bus after the first, the system has two variables but only one equation.  With $s+1$ slack buses, the rank of the Jacobian and hence the dimension of the complex solution set will be at least $s$ by Lemma 13.4.1 in \cite{Sommese2005}.
Furthermore, if a real feasible solution $U$ exists and the solution set is smooth at this point, then the local dimension of the complex and real solution sets are equal at $U$.  Therefore, since the complex solution set is positive-dimensional, so is the real set at $U$, and so infinitely many real feasible solutions exist.
\end{proof}

Although there are a variety of methods for studying positive-dimensional systems, including the enumeration of a point within each connected component \cite{basu1998new,ROUILLIER2000} and studying the critical points of the restriction to the variety of the distance function to such points \cite{AUBRY2002}, we suggest the method of moments \cite{Lasserre2006,7024950} may be more suitable for studying the feasible set of alternating-current optimal power flows.
It has been shown recently \cite{7024950} that it allows for very small errors on systems in dimension of over five thousand.

As is often the case in engineering applications, one may also be interested in the distance of a point
to the set of feasible solutions. 
Again, by considering our algebraisation, one can bound the probability this distance is small
using the theorem of Lotz \cite{Lotz2013} on the zero-set $V$ of multivariate polynomials.
This could be seen as the converse of the results of \cite{596944}.




\begin{landscape}

\begin{table*}
\vspace{0.05in}
    \centering
    \begin{tabular}{r|rrrrrrrrrrrr}
             & \multicolumn{12}{l}{\textbf{No. $|N|$ of buses}} \\  
\textbf{Method} &  3 &  4 &   5 &    6 &    7 &     8 &     9 &     10 &      11 &      12 &       13 &       14          \\ 
             \midrule
    	                             	                         B\'ezout's upper bound & 16 & 64 & 256 & 1024 & 4096 & 16384 & 65536 & 262144 & 1048576 & 4194304 & 16777216 & 67108864 \\
    	        A BKK-based upper bound &  8 & 40 & 192 &  864 & 3712 & 15488 & 63488 & 257536 & 1038336 & 4171776 & 16728064 & 67002368 \\
    	                         Theorem \ref{thm:finiteRoots} &  6 & 20 &  70 &  252 &  924 &  3432 & 12870 &  48620 &  184756 &  705432 &  2704156 & 10400600 \\
    	               \midrule
    	                   	                   	               Generic lower bound &  6 & 20 &  70 &  252 &  924 &  3432 & 12870 &  48620 &  184756 &  705432 &  2704156 & 10400600 \\
\bottomrule
    \end{tabular}
    \caption{The maximum number of steady states in a circuit with a fixed number of buses.}
    \label{tab:complete}
\end{table*}

\begin{table*}[ht!]
\begin{tabular}
{ll|rrrrrrrrrrrr}
{ \bf Instance} & { \bf Source} & 
{ \bf \rothead{No. $|N|$ of buses}} & { \bf \rothead{No. $|E|$ of branches}} & { \bf \rothead{Treewidth tw$(P)$}} &
{ \bf \rothead{No. $|X|$ of solutions}} & { \bf \rothead{min$_{x \in X}$(cost($x$))}} & { \bf \rothead{No. of min. wrt. cost}} & { \bf \rothead{avg$_{x \in X}$(cost($x$))}} & { \bf \rothead{max$_{x \in X}$(cost($x$))}} & { \bf \rothead{min$_{x \in X}$(loss($x$))}} & { \bf \rothead{No. of min. wrt. loss}} & { \bf \rothead{avg$_{x \in X}$(loss($x$))}} & { \bf \rothead{max$_{x \in X}$(loss($x$))}}\\ 
             \midrule
case2w & Bukhsh et al. \cite{6581918} & 2 & 1 & 1 & 2 & 8.42 & 1 & 9.04 & 9.66 & 0.71 & 1 & 1.02 & 1.33\\ 
case3KW & Klos and Wojcicka \cite{Klos1991} & 3 & 3 & 2 & 6 & -0.0 & 1 & 1250.0 & 1500.0 & -0.0 & 1 & -0.0 & 0.0\\ 
case3LL & Lavaei and Low \cite{5971792} & 3 & 3 & 2 & 2 & 1502.07 & 1 & 1502.19 & 1502.31 & 0.22 & 1 & 0.34 & 0.46\\ 
case3w & Bukhsh et al. \cite{6581918} & 3 & 2 & 1 & 2 & 5.88 & 1 & 5.94 & 6.01 & 0.55 & 1 & 0.58 & 0.61\\ 
case4 & McCoy et al. & 4 & 3 & 1 & 4 & 1502.51 & 1 & 1505.19 & 1507.88 & 0.11 & 1 & 2.79 & 5.48\\ 
case4ac & McCoy et al. & 4 & 4 & 2 & 6 & 2005.25 & 2 & 3337.5 & 4005.51 & 0.01 & 1 & 2.44 & 3.78\\ 
case4cyc & Bukhsh et al. \cite{6581918} & 4 & 4 & 2 & 4 & 3001.61 & 1 & 3003.56 & 3005.52 & 0.01 & 1 & 1.96 & 3.92\\ 
case4gs &  Grainger and Stevenson \cite{5491276} & 4 & 4 & 2 & 10 & 2316.37 & 1 & 3681.07 & 4595.8 & 0.37 & 1 & 27.24 & 39.72\\ 
case5w & Lesieutre et al. \cite{6120344} & 5 & 6 & 2 & 2 & 2003.57 & 1 & 2004.6 & 2005.63 & 0.47 & 1 & 1.5 & 2.53\\ 
case6ac & McCoy et al. & 6 & 6 & 2 & 23 & 4005.44 & 1 & 5574.23 & 6013.26 & 0.02 & 1 & 6.26 & 10.51\\ 
case6ac2 & McCoy et al. & 6 & 6 & 2 & 22 & 4005.32 & 1 & 5554.03 & 6012.82 & 0.02 & 1 & 5.97 & 10.21\\ 
case6b & McCoy et al. & 6 & 6 & 2 & 30 & 3005.7 & 5 & 3606.3 & 4508.28 & 0.01 & 1 & 3.9 & 5.88\\ 
case6cyc & Bukhsh et al. \cite{6581918} & 6 & 6 & 2 & 30 & 3005.7 & 4 & 3606.3 & 4508.28 & 0.01 & 1 & 3.9 & 5.88\\ 
case6cyc2 & McCoy et al. & 6 & 6 & 2 & 12 & 1506.87 & 1 & 3131.55 & 4508.24 & 0.03 & 1 & 4.15 & 6.56\\ 
case6cyc3 & McCoy et al. & 6 & 6 & 2 & 12 & 4502.42 & 1 & 4505.95 & 4508.01 & 0.02 & 1 & 3.55 & 5.61\\ 
case7 & McCoy et al. & 7 & 7 & 2 & 2 & 1667.31 & 1 & 1668.36 & 1669.41 & 0.08 & 1 & 0.26 & 0.45\\ 
case8cyc & Bukhsh et al. \cite{6581918} & 8 & 8 & 2 & 60 & 6506.11 & 1 & 8557.72 & 9511.04 & 0.01 & 1 & 4.52 & 7.84\\ 
case9 & Chow \cite{5491276} & 9 & 9 & 2 & 16 & 3504.44 & 1 & 5692.02 & 6505.02 & 0.03 & 1 & 1.37 & 1.87\\ 
case9g & McCoy et al. & 9 & 9 & 2 & 2 & 2604.41 & 1 & 4103.36 & 5602.31 & 0.08 & 1 & 0.26 & 0.45\\ 
caseK4 & McCoy et al. & 4 & 6 & 3 & 8 & 2008.62 & 1 & 3006.32 & 4005.51 & 0.01 & 1 & 4.6 & 7.03\\ 
caseK4sym & McCoy et al. & 4 & 6 & 3 & 6 & 2009.28 & 2 & 3339.29 & 4005.91 & 0.01 & 1 & 3.96 & 7.28\\ 
caseK6 & McCoy et al. & 6 & 15 & 5 & 36 & 2018.2 & 1 & 5075.47 & 6030.0 & 0.02 & 1 & 17.16 & 27.25\\ 
caseK6b & McCoy et al. & 6 & 15 & 5 & 40 & 6002.79 & 1 & 6017.01 & 6019.41 & 0.01 & 1 & 14.23 & 16.62\\ 
caseK6sym & McCoy et al. & 6 & 15 & 5 & 48 & 6003.01 & 1 & 6017.75 & 6019.86 & 0.01 & 1 & 14.75 & 16.86\\
\bottomrule
\end{tabular}
\caption{Properties of the instances tested.}
\label{tab:instances}
\end{table*}
\end{landscape}

\section{Computational Illustrations}
\label{sec:computational}

In order to illustrate Theorem~\ref{thm:finiteRoots}, 
we first present the 
maximum number of steady states in a circuit with a fixed number of buses
in Table~\label{tab:complete},
and compare it to the values of our upper bound,
B\'ezout-based upper bound,
and the BKK-based upper bound \cite{Chen2015,Mehta2015a}.
Notice that the 
maximum number of steady states in a circuit with a fixed number of buses
is achieved when $(N, E)$ is a clique.
Notice further that the generic lower bound, obtained as the number of solutions found by tracing the paths, matches the upper bound of Theorem~\ref{thm:finiteRoots} throughout Table~\ref{tab:complete}.

\begin{figure*}
\begin{verbatim}
variable_group V0, V1; variable_group U0, U1;
I0 = V0*Yv0_0 + V1*Yv0_1; I1 = V0*Yv1_0 + V1*Yv1_1;
J0 = U0*Yu0_0 + U1*Yu0_1; J1 = U0*Yu1_0 + U1*Yu1_1;
fv0 = V0 - 1.0; fv1 = I1*U1 + J1*V1 + 7.0;
fu0 = U0 - 1.0; fu1 = -I1*U1 + J1*V1 - 7.0*I;
\end{verbatim}
\caption{A Bertini encoding of ACPF on the two-bus instance of Bukhsh et al. 2012,
where the impedance of a single branch is $0.04 + 0.2i$.}
\label{fig:Bertini}
\end{figure*}

To illustrate Proposition~
\ref{thm:asymptotically-convergent},
we have enumerated the steady states using Bertini, a versatile package for homotopy-continuation
 methods by Sommese et al. \cite{Sommese2005}. See Figure~\ref{fig:Bertini}
for an example of Bertini input corresponding to the example above \eqref{BertiniEqs},
with constants $\mathtt{Yu\_i\_j}$ representing $Y_{i,j}$ and $\mathtt{Yv\_i\_j}$ representing $Y^*_{i,j}$.
The results are summarised in Table~\label{tab:instances}.

To illustrate Theorem~\ref{thm:finiteRoots} further, we present
the values of our upper bound on a collection of instances widely known in the power systems community.  
The instances are mostly available from the Test Case Archive of Optimal Power Flow (OPF) Problems with Local Optima
of Bukhsh\footnote{\url{http://www.maths.ed.ac.uk/OptEnergy/LocalOpt/}, accessed November 30th, 2014.},
while some have appeared in well-known papers, e.g. \cite{Klos1991},
and some are available in recent distributions of Matpower \cite{5491276}, a well-known benchmark.
In particular, we present the numbers of 
distinct roots of the instances. 
In all cases, where Theorem~\ref{thm:finiteRoots} applies, the number of solutions found 
by tracing the paths matches the upper bound of Theorem~\ref{thm:finiteRoots}, certifying
the completeness. In other cases, one could rely on Bertini certificates of completeness of the search.

Empirically, we observe there exists a unique global optimum in all these instances tested
with respect to the $L_1$-loss objective.
For the generation cost objective, however, there are a number of instances (case4ac,
caseK4sym, case6b, case6cyc), where the global optima are not unique. The case of caseK4sym
is a particularly good illustration, where the symmetry between two generators and
two demand nodes in a complete graph results in multiple global optima.

In order to provide material for further study of structural properties, we present tree-width
of the instances in Table~\ref{tab:instances} in column tw$(N)$.
Notice that for the well-known small instances, tree-width is 1 or 2, e.g., 1 for the instance in Figure~\ref{fig:Bertini}, and 2 for the instance of Lesieutre et al.
As the instances grow, however, this need not be the case: 
Kloks \cite{Kloks1994} shows treewidth is not bounded even in sparse random graphs, with high probability. 
In complete graphs, such as caseK4sym with tree-width 3 above, 
tree-width grows linearly in the number of vertices.

\section{Related Work}
\label{relatedwork}

There is a long history of study of the number and structure of solutions of power flows
\cite{4074830,4047153,4047382,1085093}.
\cite{4074830} considered the B\'ezout bound.
\cite{Chen2015,Mehta2015a} considered a bound based on the work of Bernstein \cite{BERNSTEIN} and Kushnirenko \cite{KUSHNIRENKO}.
\cite{1085093,4047153} derived the same expression as in Theorem 1 using intersection theory, but in the lossless AC model.
They highlight that the number of solutions in an alternating-current model with losses is an important open problem.
We note that Theorem 1 subsumes Theorem 4.1 of \cite{1085093} as a special case. 
Finally, we note that \cite{Klos1991} present a lower bound without a proof and recent papers \cite{6344759,7470293} bound certain distinguished solutions, but not all solutions; cf. \cite{Mehta2015b}.

There is also a long history of applications of homotopy-continuation methods in power systems \cite{70552,230632,260891,1406189},
although often, e.g. in \cite{260891}, the set-up of the homotopy restricted the methods to a heuristic, which could not enumerate all the solutions of the power flow \cite{6225409}.
Recently, these have attracted much interest \cite{Mehta2015a,Mehta2014,Mehta2015} 
following the work of Trias \cite{6344759,PatentTrias,7393554}.  
See \cite{Mehta2015b} for an overview.

\section{Conclusions}
\label{sec:conclusion}

We hope that the structural results provided will aid the development
of faster solvers for the related non-linear problems \cite{481635}. Arguably,
one could:
\begin{itemize}
\item By using Theorem \ref{thm:finiteRoots} in the construction of start systems for homotopy-continuation methods \cite{Wampler1993},
      allow for larger zero-dimensional systems to be studied.
\item Extend Corollary \ref{thm:positiveDimensional} to finding at least one point in each connected component \cite{ROUILLIER2000,AUBRY2002}.
\item Extend the homotopy-continuation methods to consider inequalities
within the tracing, rather than only in the filtering phase,
which could improve their computational performance considerably.
\item Develop methods for the optimal power flow problem, whose complexity
would be superpolynomial only in the tree-width and the number of buses.
\end{itemize}
The latter two may be some of the most important challenges within the analysis of circuits and systems.

\newpage
\paragraph*{Acknowledgements}
Parts of this work have been done while Tim was visiting IBM Research.
Jakub would like to thank Isaac Newton Institute for Mathematical Sciences at the University of Cambridge for their generous support for his visits.
Dhagash Mehta has kindly provided a variety of suggestions as the related work.

\bibliographystyle{abbrv}
\bibliography{tim-references}

\begin{thebibliography}{10}

\bibitem{AUBRY2002}
P.~Aubry, F.~Rouillier, and M.~S.~E. Din.
\newblock Real solving for positive dimensional systems.
\newblock {\em Journal of Symbolic Computation}, 34(6):543 -- 560, 2002.

\bibitem{1085093}
J.~Baillieul and C.~Byrnes.
\newblock Geometric critical point analysis of lossless power system models.
\newblock {\em IEEE Transactions on Circuits and Systems}, 29(11):724--737, Nov
  1982.

\bibitem{4047382}
J.~Baillieul and C.~I. Byrnes.
\newblock Remarks on the number of solutions to the load flow equations for a
  power system with electrical losses.
\newblock In {\em Decision and Control, 1982 21st IEEE Conference on}, pages
  919--924, Dec 1982.

\bibitem{4047153}
J.~Baillieul, C.~I. Byrnes, and R.~B. Washburn.
\newblock An algebraic-geometric and topological analysis of the solution to
  the load-flow equations for a power system.
\newblock In {\em Decision and Control including the Symposium on Adaptive
  Processes, 1981 20th IEEE Conference on}, pages 1312--1320, Dec 1981.

\bibitem{basu1998new}
S.~Basu, R.~Pollack, and M.-F. Roy.
\newblock A new algorithm to find a point in every cell defined by a family of
  polynomials.
\newblock In {\em Quantifier elimination and cylindrical algebraic
  decomposition}, pages 341--350. Springer Vienna, 1998.

\bibitem{bates2013numerically}
D.~Bates, J.~Hauenstein, A.~Sommese, and C.~Wampler.
\newblock {\em Numerically Solving Polynomial Systems with Bertini:}.
\newblock Software, Environments, and Tools. Society for Industrial and Applied
  Mathematics, 2013.

\bibitem{BERNSTEIN}
D.~N. Bernstein.
\newblock The number of roots of a system of equations.
\newblock {\em Funkcional. Anal. i Prilozen.}, 9(3):1--4, 1975.

\bibitem{6581918}
W.~Bukhsh, A.~Grothey, K.~McKinnon, and P.~Trodden.
\newblock Local solutions of the optimal power flow problem.
\newblock {\em IEEE Trans. Power Syst.}, 28(4):4780--4788, 2013.

\bibitem{Mehta2015}
S.~Chandra, D.~Mehta, and A.~Chakrabortty.
\newblock Equilibria analysis of power systems using a numerical homotopy
  method.
\newblock In {\em Power Energy Society General Meeting, 2015 IEEE}, pages 1--5,
  July 2015.

\bibitem{Chen2015}
T.~{Chen} and D.~{Mehta}.
\newblock {On the Network Topology Dependent Solution Count of the Algebraic
  Load Flow Equations}.
\newblock 2015.
\newblock ArXiv e-prints 1512.04987.

\bibitem{260940}
H.-D. Chiang, C.-W. Liu, P.~P. Varaiya, F.~F. Wu, and M.~G. Lauby.
\newblock Chaos in a simple power system.
\newblock {\em IEEE Transactions on Power Systems}, 8(4):1407--1417, Nov 1993.

\bibitem{fulton1998}
W.~Fulton.
\newblock {\em Intersection {{Theory}}}.
\newblock {Springer New York}, Jan. 1998.

\bibitem{7024950}
B.~Ghaddar, J.~Marecek, and M.~Mevissen.
\newblock Optimal power flow as a polynomial optimization problem.
\newblock {\em IEEE Transactions on Power Systems}, 31(1):539--546, Jan 2016.

\bibitem{230632}
S.~X. Guo and F.~M.~A. Salam.
\newblock The real homotopy-based method for computing solutions of electric
  power systems.
\newblock In {\em Circuits and Systems, 1992. ISCAS '92. Proceedings., 1992
  IEEE International Symposium on}, volume~6, pages 2737--2740 vol.6, May 1992.

\bibitem{481635}
I.~A. Hiskens.
\newblock Analysis tools for power systems-contending with nonlinearities.
\newblock {\em Proceedings of the IEEE}, 83(11):1573--1587, Nov 1995.

\bibitem{932273}
I.~A. Hiskens and R.~J. Davy.
\newblock Exploring the power flow solution space boundary.
\newblock {\em IEEE Transactions on Power Systems}, 16(3):389--395, Aug 2001.

\bibitem{Kloks1994}
T.~Kloks.
\newblock Only few graphs have bounded treewidth.
\newblock In T.~Kloks, editor, {\em Treewidth}, volume 842 of {\em Lecture
  Notes in Computer Science}, pages 51--60. Springer Berlin Heidelberg, 1994.

\bibitem{Klos1991}
A.~Klos and J.~Wojcicka.
\newblock Physical aspects of the nonuniqueness of load flow solutions.
\newblock {\em International Journal of Electrical Power and Energy Systems},
  13(5):268--276, 1991.

\bibitem{KUSHNIRENKO}
A.~Kushnirenko.
\newblock Newton polytopes and the {B}\'ezout theorem.
\newblock {\em Funct. Anal. Appl.}, 10:233--235, 1976.

\bibitem{Lasserre2006}
J.~B. Lasserre.
\newblock Convergent sdp-relaxations in polynomial optimization with sparsity.
\newblock {\em SIAM Journal on Optimization}, 17(3):822--843, 2006.

\bibitem{5971792}
J.~Lavaei and S.~Low.
\newblock Zero duality gap in optimal power flow problem.
\newblock {\em Power Systems, IEEE Transactions on}, 27(1):92 --107, feb. 2012.

\bibitem{6120344}
B.~Lesieutre, D.~Molzahn, A.~Borden, and C.~DeMarco.
\newblock Examining the limits of the application of semidefinite programming
  to power flow problems.
\newblock In {\em Communication, Control, and Computing (Allerton), 2011 49th
  Annual Allerton Conference on}, pages 1492--1499, 2011.

\bibitem{1406189}
C.-W. Liu, C.-S. Chang, J.~A. Jiang, and G.~H. Yeh.
\newblock Toward a cpflow-based algorithm to compute all the type-1 load-flow
  solutions in electric power systems.
\newblock {\em IEEE Transactions on Circuits and Systems I: Regular Papers},
  52(3):625--630, March 2005.

\bibitem{596944}
C.-W. Liu and J.~S. Thorp.
\newblock A novel method to compute the closest unstable equilibrium point for
  transient stability region estimate in power systems.
\newblock {\em IEEE Transactions on Circuits and Systems I: Fundamental Theory
  and Applications}, 44(7):630--635, Jul 1997.

\bibitem{Lotz2013}
M.~Lotz.
\newblock On the volume of tubular neighborhoods of real algebraic varieties.
\newblock {\em Proc. Amer. Math. Soc.}, 143:1875--1889, 2015.

\bibitem{260891}
W.~Ma and J.~S. Thorp.
\newblock An efficient algorithm to locate all the load flow solutions.
\newblock {\em IEEE Transactions on Power Systems}, 8(3):1077--1083, Aug 1993.

\bibitem{malajovich2007computing}
G.~Malajovich and K.~Meer.
\newblock Computing minimal multi-homogeneous {B}{\'e}zout numbers is hard.
\newblock {\em Theory of Computing Systems}, 40(4):553--570, 2007.

\bibitem{Mehta2015b}
D.~Mehta, D.~K. Molzahn, and K.~Turitsyn.
\newblock Recent advances in computational methods for the power flow
  equations.
\newblock In {\em 2016 American Control Conference (ACC)}, pages 1753--1765,
  July 2016.

\bibitem{Mehta2014}
D.~Mehta, H.~D. Nguyen, and K.~Turitsyn.
\newblock Numerical polynomial homotopy continuation method to locate all the
  power flow solutions.
\newblock {\em IET Generation, Transmission Distribution}, 10(12):2972--2980,
  2016.

\bibitem{6225409}
D.~K. Molzahn, B.~C. Lesieutre, and H.~Chen.
\newblock Counterexample to a continuation-based algorithm for finding all
  power flow solutions.
\newblock {\em IEEE Transactions on Power Systems}, 28(1):564--565, Feb 2013.

\bibitem{Mehta2015a}
D.~K. Molzahn, D.~Mehta, and M.~Niemerg.
\newblock Toward topologically based upper bounds on the number of power flow
  solutions.
\newblock In {\em 2016 American Control Conference (ACC)}, pages 5927--5932,
  July 2016.

\bibitem{MORGAN-SOMMESE}
A.~Morgan and A.~Sommese.
\newblock A homotopy for solving general polynomial systems that respects
  $m$-homogeneous structures.
\newblock {\em Appl. Math. Comput.}, 24(2):101--113, 1987.

\bibitem{ROUILLIER2000}
F.~Rouillier, M.-F. Roy, and M.~S.~E. Din.
\newblock Finding at least one point in each connected component of a real
  algebraic set defined by a single equation.
\newblock {\em Journal of Complexity}, 16(4):716--750, 2000.

\bibitem{70552}
F.~M.~A. Salam, L.~Ni, S.~Guo, and X.~Sun.
\newblock Parallel processing for the load flow of power systems: the approach
  and applications.
\newblock In {\em Decision and Control, 1989., Proceedings of the 28th IEEE
  Conference on}, pages 2173--2178 vol.3, Dec 1989.

\bibitem{SHAFAREVICH}
I.~R. Shafarevich.
\newblock {\em Basic algebraic geometry}.
\newblock Springer-Verlag, springer study edition edition, 1977.
\newblock Translated from the Russian by K. A. Hirsch; Revised printing of
  Grundlehren der mathematischen Wissenschaften, Vol. 213, 1974.

\bibitem{Sommese2005}
A.~J. Sommese and C.~W. Wampler.
\newblock {\em The numerical solution of systems of polynomials arising in
  Engineering and Science}.
\newblock World Scientific, 2005.

\bibitem{4074830}
C.~J. Tavora and O.~J.~M. Smith.
\newblock Equilibrium analysis of power systems.
\newblock {\em IEEE Transactions on Power Apparatus and Systems},
  PAS-91(3):1131--1137, May 1972.

\bibitem{70472}
J.~S. Thorp and S.~A. Naqavi.
\newblock Load flow fractals.
\newblock In {\em Decision and Control, 1989., Proceedings of the 28th IEEE
  Conference on}, pages 1822--1827 vol.2, Dec 1989.

\bibitem{6344759}
A.~Trias.
\newblock The holomorphic embedding load flow method.
\newblock In {\em Power and Energy Society General Meeting, 2012 IEEE}, pages
  1--8, July 2012.

\bibitem{PatentTrias}
A.~Trias.
\newblock System and method for monitoring and managing electrical power
  transmission and distribution networks, {US Patent 7519506 and 7979239, 2009
  and 2010}, number = {US 7519506 and 7979239}, type = {Patent}, version = {},
  location = {US},.

\bibitem{7393554}
A.~Trias and J.~L. Marín.
\newblock The holomorphic embedding loadflow method for dc power systems and
  nonlinear dc circuits.
\newblock {\em IEEE Transactions on Circuits and Systems I: Regular Papers},
  63(2):322--333, Feb 2016.

\bibitem{WAMPLER1992}
C.~W. Wampler.
\newblock Bezout number calculations for multi-homogeneous polynomial systems.
\newblock {\em Applied Mathematics and Computation}, 51(2):143 -- 157, 1992.

\bibitem{Wampler1993}
C.~W. Wampler.
\newblock An efficient start system for multi-homogeneous polynomial
  continuation.
\newblock {\em Numerische Mathematik}, 66(4):517--524, 1993/94.

\bibitem{7470293}
T.~Wang and H.~D. Chiang.
\newblock On the number of system separations in electric power systems.
\newblock {\em IEEE Transactions on Circuits and Systems I: Regular Papers},
  63(5):661--670, May 2016.

\bibitem{5491276}
R.~Zimmerman, C.~Murillo-S{\'a}nchez, and R.~Thomas.
\newblock Matpower: Steady-state operations, planning, and analysis tools for
  power systems research and education.
\newblock {\em Power Systems, IEEE Transactions on}, 26(1):12--19, 2011.

\end{thebibliography}

\end{document}